\documentclass[a4paper]{amsart}
\usepackage{graphicx}
\usepackage{amsmath}
\usepackage{amssymb}
\usepackage{oldgerm}
\usepackage[all]{xy}

\newcommand{\Hom}{\operatorname{Hom}\nolimits}
\renewcommand{\Im}{\operatorname{Im}\nolimits}

\newcommand{\pd}{\operatorname{pd}\nolimits}

\newcommand{\Ext}{\operatorname{Ext}\nolimits}

\renewcommand{\H}{\operatorname{H}\nolimits}
\newcommand{\R}{\operatorname{\bf{R}}\nolimits}

\newcommand{\m}{\operatorname{\mathfrak{m}}\nolimits}

\newcommand{\cx}{\operatorname{cx}\nolimits}

\newcommand{\T}{\operatorname{\mathcal{T}}\nolimits}
\newcommand{\C}{\operatorname{\mathcal{C}}\nolimits}
\newcommand{\D}{\operatorname{\mathcal{D}}\nolimits}

\newcommand{\thick}{\operatorname{thick}\nolimits}
\newcommand{\s}{\operatorname{\Sigma}\nolimits}

\newtheorem{theorem}{Theorem}[section]

\newtheorem{corollary}[theorem]{Corollary}

\newtheorem{lemma}[theorem]{Lemma}

\theoremstyle{definition}

\theoremstyle{definition}

\theoremstyle{definition}

\theoremstyle{definition}

\theoremstyle{definition}

\theoremstyle{definition}

\theoremstyle{remark}

\theoremstyle{definition}

\theoremstyle{definition}

\begin{document}

\title{On complexes of finite complete intersection dimension}
\author{Petter Andreas Bergh}
\address{Institutt for matematiske fag \\
  NTNU \\ N-7491 Trondheim \\ Norway}
\email{bergh@math.ntnu.no}

\thanks{The author was supported by NFR Storforsk grant no.\
167130}

\subjclass[2000]{13D25, 18E30, 18G10}

\keywords{Finite complete intersection dimension, complexity, virtually small complexes}

\maketitle

\begin{abstract}
We study complexes of finite complete intersection dimension in the derived category of a local ring. Given such a complex, we prove that the thick subcategory it generates contains complexes of all possible complexities. In particular, we show that such a complex is virtually small, answering a question raised by Dwyer, Greenlees and Iyengar.
\end{abstract}

\section{Introduction}

In \cite{DwyerGreenleesIyengar}, the authors raised the question
whether every nonzero homologically finite complex of finite
complete intersection dimension over a local ring is virtually
small. In other words, given such a ring $A$ and such a complex $M
\in D(A)$, is the intersection
$$\thick_{D(A)}(M) \cap \thick_{D(A)}(A)$$
nonzero? We give an affirmative answer to this question, by showing
that the thick subcategory generated by $M$ contains a nonzero
complex of complexity zero. In fact, we show that if the complexity
of $M$ is $c$, then $\thick_{D(A)}(M)$ contains a nonzero complex of
complexity $t$ for every $0 \le t \le c$. The homologically finite
complexes of complexity zero are precisely the complexes of finite
projective dimension. Moreover, a homologically finite complex
belongs to $\thick_{D(A)}(A)$ if and only if its projective
dimension is finite. Thus, the results mentioned indeed settle the
above question.

\section{Notation and terminology}

Let $(A, \m, k)$ be a local (meaning commutative Noetherian local)
ring, and denote by $D(A)$ the derived category of (not necessarily
finitely generated) $A$-modules. A complex
$$M \colon \cdots \to M_{n+1} \to M_n \to M_{n-1} \to \cdots$$
in $D(A)$ is \emph{bounded below} if $M_n=0$ for $n \ll 0$, and
\emph{bounded above} if $M_n=0$ for $n \gg 0$. The complex is
\emph{bounded} if it is both bounded below and bounded above, and
\emph{finite} if it is bounded and degreewise finitely generated.
The \emph{homology} of $M$, denoted $\H (M)$, is the complex with
$\H (M)_n = \H_n (M)$, and with trivial differentials. When $\H
(M)$ is finite, then $M$ is said to be \emph{homologically
finite}.

As shown for example in \cite{Roberts}, when the complex $M$ is
homologically finite, then it has a minimal free resolution. Thus,
there exists a quasi-isomorphism $F \simeq M$, where $F$ is a
bounded below complex
$$\cdots \to F_{n+1} \xrightarrow{d_{n+1}} F_n \xrightarrow{d_n} F_{n-1} \to \cdots$$
of finitely generated free modules, and where $\Im d_n \subseteq \m
F_{n-1}$. This resolution is unique up to isomorphism, and so for
each integer $n$ the rank of the free module $F_n$ is a well defined
invariant of $M$. This is the $n$th \emph{Betti number} $\beta_n
(M)$ of $M$, and the corresponding generating function
$\sum_{n=0}^{\infty} \beta_n (M)t^n$ is the Poincar{\'e} series
$P(M,t)$ of $M$. The \emph{complexity} of $M$, denoted $\cx M$, is
defined as
$$\cx M \stackrel{\text{def}}{=} \inf \{ t \in \mathbb{N} \cup
\{ 0 \} \mid \exists a \in \mathbb{R} \text{ such that } \beta_n (M)
\le an^{t-1} \text{ for } n \gg 0 \}.$$ The complexity of a
homologically finite complex is not necessarily finite. In fact, by
a theorem of Gulliksen (cf.\ \cite{Gulliksen}), finiteness of the
complexities of all homologically finite complexes in $D(A)$ is
equivalent to $A$ being a complete intersection ring.

Suppose that our complex $M$ is homologically finite, with a minimal
free resolution $F$. Given a complex $N \in D(A)$, the complex
$\Hom_A (F,N)$ is denoted by $\R \Hom_A(M,N)$. Up to
quasi-isomorphism, this complex is well defined, hence so is the
cohomology group
$$\Ext_A^n(M,N) \stackrel{\text{def}}{=} \H_{-n} \left ( \R
\Hom_A(M,N) \right )$$ for every integer $n$. The \emph{projective
dimension} of $M$, denoted $\pd_A M$, is defined as
$$\sup \{ n \mid \Ext_A^n(M,k) \neq 0 \},$$
which is the same as the supremum of all integers $n$ such that
$F_n$ is nonzero. Namely, since the complex $F$ is minimal, the
differentials in the complex $\Hom_A(F,k)$ are trivial, and $\beta_n
(M) = \dim_k \Ext_A^n(M,k)$ for all $n$. In particular, the
projective dimension of $M$ is finite if and only if its complexity
is zero.

The derived category $D(A)$ is triangulated, the suspension functor
$\s$ being the left shift of a complex. Given complexes $M$ and $N$
as above, for each $n$ we may identify the cohomology group
$\Ext_A^n(M,N)$ with $\Hom_{D(A)}(M, \s^nN)$. For complexes $X$ and
$Y$ in $D(A)$, denote the graded $A$-module $\oplus_{n=0}^{\infty}
\Ext_A^n(X,Y)$ by $\Ext_A^*(X,Y)$. Using composition of maps in
$D(A)$, the graded $A$-module $\Ext_A^*(M,M)$ becomes a ring, and
$\Ext_A^*(M,N)$ becomes a graded right $\Ext_A^*(M,M)$-module.

We end this section by recalling the notion of nearness in an
arbitrary triangulated category $\T$ with a suspension functor $\s$.
A subcategory of $\T$ is \emph{thick} if it is a full triangulated
subcategory closed under direct summands. Now let $\C$ and $\D$ be
subcategories of $\T$. We denote by $\thick^1_{\T} ( \C )$ the full
subcategory of $\T$ consisting of all the direct summands of finite
direct sums of shifts of objects in $\C$. Furthermore, we denote by
$\C \ast \D$ the full subcategory of $\T$ consisting of objects $M$
such that there exists a distinguished triangle
$$C \to M \to D \to \s C$$
in $\T$, with $C \in \C$ and $D \in \D$. Now for each $n \ge 2$,
define inductively $\thick^n_{\T} ( \C )$ to be $\thick_{\T}^1 \left
( \thick^{n-1}_{\T} ( \C ) \ast \thick^1_{\T} ( \C ) \right )$, and
denote $\bigcup_{n=1}^{\infty} \thick^n_{\T} ( \C )$ by $\thick_{\T}
( \C )$. This is the smallest thick subcategory of $\T$ containing
$\C$.

\section{Complexes of finite CI-dimension}

Let $A$ be a local ring. Recall that a \emph{quasi deformation} of
$A$ is a diagram $A \to R \leftarrow Q$ of local homomorphisms, in
which $A \to R$ is faithfully flat, and $R \leftarrow Q$ is
surjective with kernel generated by a regular sequence. A
homologically finite complex $M \in D(A)$ has finite \emph{complete
intersection dimension} if there exists such a quasi deformation for
which $\pd_Q (R \otimes_A M)$ is finite. From now on, we write
``CI-dimension" instead of ``complete intersection dimension".

The notion of CI-dimension was first introduced for modules in
\cite{AvramovGasharovPeeva}. The terminology reflects the fact that
a local ring is a complete intersection precisely when all its
finitely generated modules have finite CI-dimension. The same holds
if we replace ``modules" with ``homologically finite complexes".

In order to prove the main result, we need the following lemma. It
shows that finite CI-dimension is preserved in thick
subcategories.

\begin{lemma}\label{CIthick}
Let $A$ be a local ring, and let $M \in D(A)$ be a homologically
finite complex of finite CI-dimension. Then every complex in
$\thick_{D(A)}(M)$ has finite CI-dimension.
\end{lemma}

\begin{proof}
Let $A \to R \twoheadleftarrow Q$ be a quasi-deformation of $A$ such
that $\pd_Q (R \otimes_A M)$ is finite. We show by induction that
for all $n \ge 1$, every complex $X \in \thick_{D(A)}^n(M)$
satisfies $\pd_Q (R \otimes_A X) < \infty$. If $X$ is a direct
summand of finite direct sums of shifts of $M$, then this clearly
holds. Therefore $\pd_Q (R \otimes_A X)$ is finite for all complexes
$X \in \thick_{D(A)}^1(M)$. Next, suppose the claim holds for all
$1, \dots, n$, and let $X$ be a complex in $\thick_{D(A)}^n(M) \ast
\thick_{D(A)}^1(M)$. Then there exists a triangle
$$Y \to X \to Z \to \s Y$$
in $D(A)$, where $Y$ and $Z$ are complexes in $\thick_{D(A)}^n(M)$
and $\thick_{D(A)}^1(M)$, respectively. By induction, $\pd_Q (R
\otimes_A Y)$ and $\pd_Q (R \otimes_A Z)$ are both finite, hence so
is $\pd_Q (R \otimes_A X)$. Since
$$\thick_{D(A)}^{n+1}(M) = \thick_{D(A)}^1 \left (
\thick_{D(A)}^n(M) \ast \thick_{D(A)}^1(M) \right ),$$ the proof is
complete.
\end{proof}

Next, we prove the main result; the thick subcategory generated by a
complex of finite CI-dimension contains complexes of all possible
complexities.

\begin{theorem}\label{main}
Let $A$ be a local ring, and let $M \in D(A)$ be a nonzero
homologically finite complex of finite CI-dimension. Then for every
$0 \le t \le \cx M$, there exists a nonzero complex in
$\thick_{D(A)}(M)$ of complexity $t$.
\end{theorem}

\begin{proof}
The proof is by induction on the complexity $c$ of $M$. If $c=0$,
then there is nothing to prove, so suppose that $c$ is nonzero. We shall construct a nonzero complex in
$\thick_{D(A)}(M)$ of complexity $c-1$.

Let $A \to R \twoheadleftarrow Q$ be a quasi-deformation of $A$
such that $\pd_Q (R \otimes_A M)$ is finite, and let $Z_R( R
\otimes_A M)$ be the graded $R$-subalgebra of $\Ext_R^*(R
\otimes_A M, R \otimes_A M)$ generated by the central elements in
degree two. By \cite[Corollary 5.1]{AvramovSun}, the $Z_R( R
\otimes_A M)$-module $\Ext_R^*(R \otimes_A M, R \otimes_A k)$ is
Noetherian, where $k$ is the residue field of $A$. Now let $Z_A
(M)$ be the graded $A$-subalgebra of $\Ext_A^*(M,M)$ generated by
the central elements in degree two. Then by \cite[Theorem
4.9]{AvramovGasharovPeeva}, we may identify $Z_R( R \otimes_A M)$
with $R \otimes_A Z_A(M)$, and so by faithfully flat descent the
$Z_A(M)$-module $\Ext_A^*(M,k)$ is Noetherian (see also \cite[Section 7]{AvramovIyengar}).

By \cite[Lemma 2.5]{BIKO}, there exists a positive degree element
$\eta \in Z_A(M)$ with the property that scalar multiplication
$$\Ext_A^n(M,k) \xrightarrow{\cdot \eta} \Ext_A^{n+ |\eta|}(M,k)$$
is injective for $n \gg 0$. This element corresponds to a map $M
\xrightarrow{f_{\eta}} \s^{|\eta|}M$ in $D(A)$, and completing this
map we obtain a triangle
$$M \xrightarrow{f_{\eta}} \s^{|\eta|}M \to K \to \s M$$
in $\thick_{D(A)}(M)$. Note that the object $K$ is nonzero; if
not, then $M$ would be isomorphic to $\s^{|\eta|}M$, and this is
impossible. The triangle induces a long exact sequence
$$\cdots \to \Ext_A^n(K,k) \to \Ext_A^{n-|\eta|}(M,k) \xrightarrow{\cdot
(-1)^n \eta} \Ext_A^n(M,k) \to \Ext_A^{n+1}(K,k) \to \cdots$$ in
cohomology, hence for some integer $n_0$ the equality
$\beta_{n+1}(K) = \beta_n(M) - \beta_{n-|\eta|}(M)$ holds for $n \ge
n_0$. The Poincar{\'e} series of $K$ is then given by
\begin{eqnarray*}
P(K,t) & = & \sum_{n=0}^{\infty} \beta_n(K)t^n \\
& = & \sum_{n=0}^{n_0} \beta_n(K)t^n + \sum_{n=n_0+1}^{\infty}
\beta_n(K)t^n \\
& = & \sum_{n=0}^{n_0} \beta_n(K)t^n + \sum_{n=n_0+1}^{\infty}
\left ( \beta_{n-1}(M) - \beta_{n-|\eta|-1}(M) \right ) t^n \\
& = & g(t) + (t- t^{|\eta|+1}) P(M,t) \\
& = & g(t) + t(1+t+ \cdots + t^{|\eta|-1})(1-t)P(M,t),
\end{eqnarray*}
where $g(t)$ is a polynomial in $\mathbb{Z}[t]$.

By \cite[Corollary 3.10(iv)]{Sather-Wagstaff}, the Poincar{\'e}
series $P(M,t)$ of $M$ is rational, hence so is $P(K,t)$. Therefore, the complexities of $M$ and $K$ equal the orders of the poles at $t=1$ of $P(M,t)$ and $P(K,t)$, respectively. Consequently, the complexity of $K$ is $\cx M -1$. By Lemma \ref{CIthick}, the complex $K$, being an object in
$\thick_{D(A)}(M)$, has finite CI-dimension. By induction, for every $0 \le t \le \cx K$, there exists a nonzero complex in $\thick_{D(A)}(K)$ of complexity $t$. This completes the proof.
\end{proof}

Using \cite[3.8]{DwyerGreenleesIyengar}, the following corollary
follows immediately from Theorem \ref{main}. It settles the
question, raised in \cite{DwyerGreenleesIyengar}, whether every
nonzero homologically finite complex of finite CI-dimension over a
local ring is virtually small.

\begin{corollary}\label{virtuallysmall}
Let $A$ be a local ring, and let $M \in D(A)$ be a nonzero
homologically finite complex of finite CI-dimension. Then
$$\thick_{D(A)}(M) \cap \thick_{D(A)}(A)$$
is nonzero.
\end{corollary}

We also obtain the following criterion for a local ring to be
Gorenstein.

\begin{corollary}\label{Gorenstein}
Let $A$ be a local ring, and suppose that for every homologically
finite complex $M \in D(A)$ the thick subcategory $\thick_{D(A)}(M)$
contains a nonzero complex of finite CI-dimension. Then $A$ is
Gorenstein.
\end{corollary}

\begin{proof}
Follows immediately from Corollary \ref{virtuallysmall} and
\cite[Theorem 9.11]{DwyerGreenleesIyengar}.
\end{proof}

\section*{Acknowledgements}

I would like to thank Srikanth Iyengar for valuable
comments on this paper.

\end{document}